\def\timestamp{%
Time-stamp: <zero-dim-F.tex: Thursday 09-12-2021 at 13:49:21 (cet)>}
\def\stripname Time-stamp: <#1 #2>{#2}
\edef\filedate{\expandafter\stripname\timestamp}

\documentclass[a4paper]{amsart}

\newcommand\omegaseq[1]{\langle {#1}_n:n\in\omega\rangle}
\newcommand\Omegaseq[2][1]{\langle {#2}_\alpha:\alpha\in\omega_{#1}\rangle}
\newcommand\omegatwodelta{(\omega_2)_\delta}
\newcommand\omegatwoodelta{(\omega_2+1)_\delta}
\newcommand\orpr[2]{\langle{#1},{#2}\rangle}
\newcommand\bigorpr[2]{\bigl<{#1},{#2}\bigr>}
\newcommand\preim{^\gets}

\newcommand\cl{\operatorname{cl}}
\newcommand\0{\mathbf{0}}
\newcommand\calB{\mathcal{B}}
\newcommand\calI{\mathcal{I}}
\newcommand\calJ{\mathcal{J}}
\newcommand\calU{\mathcal{U}}
\newcommand\calV{\mathcal{V}}

\DeclareMathSymbol\A0{AMSb}{`A}
\DeclareMathSymbol\B0{AMSb}{`B}
\DeclareMathSymbol\HH0{AMSb}{`H}
\DeclareMathSymbol\K0{AMSb}{`K}
\DeclareMathSymbol\LL0{AMSb}{`L}
\DeclareMathSymbol\N0{AMSb}{`N}
\DeclareMathSymbol\Q0{AMSb}{`Q}
\DeclareMathSymbol\R0{AMSb}{`R}
\newcommand\betaN{\beta\N}
\newcommand\Nstar{\N^*}

\DeclareMathSymbol\restr  \mathbin{AMSa}{"16}
\DeclareMathSymbol\le     \mathrel{AMSa}{"36}
\DeclareMathSymbol\ge     \mathrel{AMSa}{"3E}

\newcommand\cee{\mathfrak{c}}

\DeclareSymbolFont{cmmib}{OML}{cmm}{b}{it}
\DeclareSymbolFont{cmbsy}{OMS}{cmsy}{b}{n}
\DeclareMathSymbol\vq0{cmmib}{`q}
\DeclareMathSymbol\namef 0{cmmib}{`f}
\DeclareMathSymbol\namex 0{cmmib}{`x}
\DeclareMathSymbol\nameI 0{cmbsy}{`I}

\newcommand\axiom{\mathsf}
\newcommand\CH{\axiom{CH}}
\newcommand\ZFC{\axiom{ZFC}}

\theoremstyle{plain}

\newtheorem{theorem}{Theorem}[section]

\newtheorem{lemma}[theorem]{Lemma}

\theoremstyle{definition}
\newtheorem{example}{Example}

\theoremstyle{remark}
\newtheorem{question}{Question}

\usepackage{amsrefs}
\newcommand\zbl[1]{ \allowbreak zbMATH~#1}

\begin{document}
\title[An $F$-space]
      {A zero-dimensional $F$-space\\ that is not strongly zero-dimensional}

\author[A. Dow]{Alan Dow}
\address{Department of Mathematics\\
         UNC-Charlotte\\
         9201 University City Blvd. \\
         Charlotte, NC 28223-0001}
\email{adow@uncc.edu}
\urladdr{https://webpages.uncc.edu/adow}

\author[K. P. Hart]{Klaas Pieter Hart}

\address{Faculty EEMCS\\TU Delft\\
         Postbus 5031\\2600~GA {} Delft\\the Netherlands}
\email{k.p.hart@tudelft.nl}
\urladdr{http://fa.ewi.tudelft.nl/\~{}hart}

\date{\filedate}



\begin{abstract}
We present an example of a zero-dimensional $F$-space that
is not strongly zero-dimensional.
\end{abstract}

\maketitle

\section*{Introduction}

In this paper we present an example of a zero-dimensional $F$-space that
is not strongly zero-dimensional.
We recall that a space is \emph{zero-dimensional} if it is a $T_1$-space and 
its clopen subsets form a base for the topology.
The fastest way to define a space to be \emph{strongly zero-dimensional} 
is by demanding that its \v{C}ech-Stone compactification is zero-dimensional.

The question whether zero-dimensionality implies strong zero-dimensionality
has a long history, a summary of which can be found 
in~\cite{MR1039321}*{Section~6.2}.
There are by now many examples of zero-dimensional spaces that are 
not strongly zero-dimensional, even metrizable ones, see~\cite{MR227960},
but the authors are not aware of an $F$-space of this nature.

The question whether there is a $F$-space example was making the rounds
already in the 1980s but it seems to have been asked explicitly only
a few years ago on MathOverFlow, see~\cite{flow239324}.
Recently Ali Reza Olfati raised the question with the first author
in a different context.

In section~\ref{sec.prelim} we give proper definitions of the notions mentioned
above and indicate why it may seem reasonable, but also illusory, 
to expect that zero-dimensional $F$-spaces are strongly zero-dimensional.

In section~\ref{sec.construct} we construct the example and in 
section~\ref{sec.variations} we discuss some variations;
the example can have arbitrary large covering dimension and its
\v{C}ech-Stone remainder can be an indecomposable continuum.

\section{Preliminaries}\label{sec.prelim}

In the introduction we defined zero-dimensional spaces as $T_1$-spaces
in which the clopen sets constitute a base for the open sets and
strong zero-dimensionality by requiring that the 
\v{C}ech-Stone compactification is zero-dimensional.

The latter is a characterization of strong zero-dimensionality.
The real definition is akin to the large inductive dimension:
a Tychonoff space $X$~is strongly zero-dimensional if any two 
completely separated sets are separated by a clopen set, that is,
if $A$ and $B$ are such that there is a continuous function $f:X\to[-1,1]$
with $f[A]=\{-1\}$ and $f[B]=\{1\}$ then there is a clopen set~$C$
such that $A\subseteq C$ and $C\cap B=\emptyset$.
One could reformulate the latter conclusion as: there is a continuous
function $c:X\to\{-1,1\}$ such that $c[A]=\{-1\}$ and $c[B]=\{1\}$.
It is not hard to show that this is equivalent to $\beta X$ being 
zero-dimensional. 

Furthermore, for normal spaces strong zero-dimensionality is characterized
by the `normal' sounding 
``disjoint closed sets are contained in complementary clopen sets''.

\bigskip
There are many characterizations of $F$-spaces, 
see~\cite{MR0407579}*{Theorem~14.25},
each of which deserves to be taken as the definition but we take the one that 
at first glance seems quite close to strong zero-dimensionality; it is
number~(5) in the theorem referred to above.
A Tychonoff space~$X$ is an $F$-space iff for every continuous function
$f:X\to\R$ there is another continuous function $k:X\to\R$ with the property 
that $f=k\cdot|f|$; so $k$~is constant on the 
sets $\{x:f(x)>0\}$ and $\{x:f(x)<0\}$ with values~$1$ and~$-1$ respectively.
Although $k$~does seem to act like the function~$c$ in our definition
of strong zero-dimensionality, it does not.

In fact there are (compact) connected $F$-spaces, 
for example $\beta\R^+\setminus\R^+$, where $\R^+=\{x\in\R:x\ge0\}$,
see~\cite{MR0407579}*{14.27} or~\cite{MR1229130}.
In such spaces the function~$k$ takes on all values in the interval~$[-1,1]$
on the set~$\{x:f(x)=0\}$, which apparently need not be as thin as we have 
come to expect from Calculus;
in an $F$-space sets like $\{x:f(x)>0\}$ and $\{x:f(x)<0\}$ are actually
very far apart.

\bigskip
Our notation is standard, see \cite{MR1039321} and \cite{MR0407579}
for topological notions, and~\cite{MR597342} for Set Theory.

\section{A zero-dimensional $F$-space that is not strongly zero-dimensional}
\label{sec.construct}

The construction in this section is inspired by an answer
to a question on MathOverFlow, see~\cite{flow93719}, 
which in turn was inspired by Dowker's example~M in~\cite{MR86286}. 
The latter is a subspace of~$\omega_1\times[0,1]$; the example
on MathOverFlow is a quotient of~$\omega_1\times\A$, where $\A$~is
Alexandroff's split interval.

\medskip
We replace the ordinal space~$\omega_1$ by the $G_\delta$-modification
of the ordinal space~$\omega_2$, which we denote~$\omegatwodelta$;
likewise $\omegatwoodelta$ denotes the $G_\delta$-modification of~$\omega_2+1$.
We replace~$\A$ by the split interval over a suitable ordered continuum.

We shall use an ordered continuum~$K$ with a dense subset~$D$ that can be
enumerated as $\Omegaseq[2]{d}$ in such a way that every tail set
$T_\alpha=\{d_\beta:\beta\ge\alpha\}$ is dense in~$K$.

\begin{example}
If $\CH$ fails then we can take $K=[0,1]$ and, like Dowker did,
choose $\aleph_2$~many distinct cosets of~$\Q$, 
say $\{\Q_\alpha:\alpha\in\omega_2\}$, and enumerate their union~$D$ 
as~$\Omegaseq[2]{d}$ in such a way that 
$\langle d_{\omega\cdot\alpha+n}:n\in\omega\rangle$ 
enumerates~$\Q_\alpha\cap(0,1)$.
\end{example}

\begin{example}
For a $\ZFC$ example let $M$ be the linearly ordered sum
$\omega_2^\star+\{\0\}+\omega_2$,
where $\omega_2^\star$ denotes~$\omega_2$ with its order reversed.
Following~\cite{zbMATH02646353} we let $L=M_{\0}(\omega)$,
that is, the set $\{x\in M^\omega:\{m:x_n\neq\0\}$~is finite$\}$,
ordered lexicographically.
It is elementary to verify that the linear order is dense, in fact every 
interval has cardinality~$\aleph_2$, and has no smallest or largest element.

We let $K$ be the Dedekind completion of~$L$
(see~\cite{zbMATH02615204}*{Kap.~IV, \S\,5}), that is, 
the set of initial segments that have no maximum 
(including $\emptyset$ and~$L$ as minimum and maximum respectively), 
ordered by inclusion.
Then $K$~is an ordered continuum and the set~$L$ itself serves as the desired 
dense set, under any enumeration.
\end{example}

We need the following Lemma, which is a variation of a result of 
Van~Douwen,
see~\cite{MR1039321}*{Problem 3.12.20.(c)}.

\begin{lemma}\label{lemma.eric}
Let $X$ be a compact Hausdorff space.
The product $\omegatwodelta\times X$ is $C$-embedded in 
$\omegatwoodelta\times X$.
\end{lemma}

\begin{proof}
Let $f:\omegatwodelta\times X\to\R$ be continuous.

Take $\alpha\in\omega_2$ of cofinality~$\aleph_1$.
For every $x\in X$ and $n\in\omega$ one can find $\beta(x,n)<\alpha$ and 
an open set~$U(x,n)$ in~$X$ such that $x\in U(x,n)$ and
$$
f\bigl[(\beta(x,n),\alpha]\times U(x,n)\bigr]\subseteq
 \bigl(f(\alpha,x)-2^{-n},f(\alpha,x)+2^{-n}\bigr)
$$
By compactness we can take a finite subcover $\{U(x,n):x\in F_n\}$
of the cover $\{U(x,n):x\in X\}$.
Let $\beta_n=\max\{\beta(x,n):x\in F_n\}$, then for all $x\in X$ 
and $\gamma\in(\beta_n,\alpha]$ we have 
$\bigl|f(\gamma,x)-f(\alpha,x)\bigr|<2^{-n+1}$.

Next let $\beta_\alpha=\sup\{\beta_n:n\in\omega\}$, then $\beta_\alpha<\alpha$
and $f$~is constant on each horizontal line $(\beta_\alpha,\alpha]\times\{x\}$.

\smallskip
The Pressing-Down Lemma now gives us a single~$\beta$ such that $f$~is
constant on~$(\beta,\omega_2)\times\{x\}$ for all~$x$.
Those constant values give us our continuous extension of~$f$
to $\omegatwoodelta\times X$. 
\end{proof}

The rest of the section is devoted to the construction of our $F$-space.

\subsection*{Split intervals}
Using the continuum~$K$ and the dense set $\{d_\alpha:\alpha\in\omega_2\}$
we create a sequence $\langle K_\alpha:\alpha\le\omega_2\rangle$ of ordered 
compacta, as follows:
$$
K_\alpha=\{\orpr xi\in K\times 2: \text{if }
x\notin T_\alpha \text{ then }i=0\}
$$
ordered lexicographically (reading from left to right).
Thus $K_\alpha$~is a split interval over~$K$, where all points~$d_\beta$
with $\beta\ge\alpha$ are split in two; if $\alpha=\omega_2$ then no points
are split and $K_{\omega_2}$ is just $K$~itself.

There are obvious maps $q_{\alpha,\beta}:K_\alpha\to K_\beta$ when $\alpha<\beta$,
defined by
\begin{align*}
q_{\alpha,\beta}(x,i)&=\orpr x0 \text{ when }x\notin\{d_\gamma:\gamma\ge\beta\}\\
q_{\alpha,\beta}(d_\gamma,i)&=\orpr{d_\gamma}i\text{ when }\gamma\ge\beta.
\end{align*}
We abbreviate the maps $q_{0,\alpha}$ by $q_\alpha$.

If $\alpha<\omega_2$ then $K_\alpha$ is zero-dimensional.
Here is where we use that every tail set~$T_\alpha$ 
is dense in~$K$.
This implies that the family $\calB_\alpha$ of all clopen intervals of the form
$\bigl[\min K,  \orpr e0\bigr]$,
$\bigl[\orpr d1,\orpr e0\bigr]$, and
$\bigl[\orpr d1,\max K\bigr]$, 
where $d,e\in T_\alpha$, is base for the topology of~$K_\alpha$.
As $K_\alpha$ is compact it is strongly zero-dimensional as well.
 
For later use: the intervals in $\calB_\alpha$ belong to $\calB_\beta$
when $\beta\le\alpha$ (when suitably interpreted) and
if $I\in\calB_\alpha$ is such an interval then it 
satisfies $I=q_{\beta,\alpha}\preim\bigl[q_{\beta,\alpha}[I]\bigr]$ 
whenever $\beta\le\alpha$.

\subsection*{Using compactifications}

To get to our $F$-space we take, for every~$\alpha\le\omega_2$,
the \v{C}ech-Stone compactification $\beta(\omega\times K_\alpha)$
of the product~$\omega\times K_\alpha$; we let $\K_\alpha$ denote
this compactification and $X_\alpha$ denotes
the remainder $(\omega\times K_\alpha)^*$.
The maps~$q_{\alpha,\beta}$ induce maps from~$\K_\alpha$ to~$\K_\beta$ when
$\alpha<\beta$; we denote these by~$\vq_{\alpha,\beta}$, 
and~$\vq_\alpha=\vq_{0,\alpha}$ of course.

If $\alpha<\omega_2$ then the product $\omega\times K_\alpha$ is 
strongly zero-dimensional because $K_\alpha$~is; 
this implies that $\K_\alpha$ and~$X_\alpha$ are zero-dimensional and
hence, by compactness, strongly zero-dimensional as well.
Furthermore, by~\cite{MR0407579}*{14.27}, every~$X_\alpha$~is an $F$-space,
including for $\alpha=\omega_2$.

We consider the product $\omegatwodelta\times\K_0$ and the union
$$
\K=\bigcup\bigl\{\{\alpha\}\times\K_\alpha:\alpha<\omega_2\}
$$
as well as $\omegatwoodelta\times\K_0$ 
and $\K^+=\K\cup(\{\omega_2\}\times \K_{\omega_2})$.

Our example will be the union of the remainders:
$$
X=\bigcup\bigl\{\{\alpha\}\times X_\alpha:\alpha<\omega_2\}
$$
and we also use $X^+=X\cup(\{\omega_2\}\times X_{\omega_2})$.

\subsection*{A quotient map and the topology}

We define $\vq:\omegatwoodelta\times\K_0\to\K^+$ by combining the
maps $\vq_\alpha$:
$$
\vq(\alpha,x) = \bigorpr\alpha{\vq_\alpha(x)}
$$
We give $\K^+$ the quotient topology determined by~$\vq$ and the product 
topology on~$\omegatwoodelta\times\K_0$.
We show that $\vq$~is a perfect map.
The fibers of~$\vq$ are clearly compact so we must show that $\vq$~is closed.

To begin note that for each~$\alpha$ the set~$\{\alpha\}\times\K_\alpha$
is closed and the map $\vq_\alpha:\K_0\to\K_\alpha$ is a closed map,
so that the quotient topology on~$\{\alpha\}\times\K_\alpha$ is its
normal topology.
Also, if $\alpha$~has countable cofinality then 
$\{\alpha\}\times\K_0$~is clopen in the product, hence 
so is~$\{\alpha\}\times\K_\alpha$ in~$\K^+$.

Next let $\alpha$ be of uncountable cofinality, take $x\in\K_\alpha$ and 
an open set~$O$ in $\omegatwoodelta\times\K_0$ such that 
$\vq\preim(\alpha,x)=\{\alpha\}\times\vq_\alpha\preim(x)\subseteq O$. 
By compactness there are an open set~$V$ in~$\K_0$ and $\beta<\alpha$
such that 
$$
\{\alpha\}\times\vq_\alpha\preim(x)\subseteq(\beta,\alpha]\times V\subseteq O
$$
Because $\vq_\alpha:\K_0\to\K_\alpha$ is closed there is an open set~$U$ 
in~$\K_\alpha$ such that $\vq_\alpha\preim[U]\subseteq V$.
Then $(\beta,\alpha]\times\vq_\alpha\preim[U]$ is an open subset 
of~$\omegatwoodelta\times\K_0$.
For $\gamma\in(\beta,\alpha)$ we have 
$\vq_\alpha=\vq_{\gamma,\alpha}\circ\vq_\gamma$,
hence $\vq_\alpha\preim[U]=\vq_\gamma\preim\bigl[\vq_{\gamma,\alpha}\preim[U]\bigr]$.

It follows that $\vq\preim[W]=(\beta,\alpha]\times\vq_\alpha\preim[U]$, where
$$
W=\bigcup\bigl\{\{\gamma\}\times\vq_{\gamma,\alpha}\preim[U]:
                \beta<\gamma\le\alpha\bigr\}
$$
The set $W$ is therefore open and $\vq\preim[W]\subseteq O$.

This argument also shows that $\K$~is zero-dimensional because 
if $\alpha<\omega_2$ the set~$U$ can be taken to be a clopen set and the 
resulting set~$W$ is clopen as well.

\smallskip
Thus far we have topologized $\K^+$ and hence $X^+$ and we have shown that
$X$~is zero-dimensional.
We now turn to showing that $X^+$~is an $F$-space and that $X$~is $C$-embedded
in~$X^+$.
This will show that $\beta X=\beta X^+$, 
hence $X$~is an $F$-space as well (by~\cite{MR0407579}*{14.25})
but not strongly zero-dimensional because the one-dimensional
space~$X_{\omega_2}$ is a subspace of~$\beta X$; we establish the 
one-dimensionality of~$X_{\omega_2}$ in the next section.

\subsection*{$C$-embedding}

To show that $X$ is $C$-embedded in~$X^+$ we let $f:X\to\R$ be continuous
and apply the proof of Lemma~\ref{lemma.eric} to 
$f\circ\vq:\omegatwodelta\times X_0\to\R$ to find an $\alpha<\omega_2$
such that $f\circ\vq$~is constant on $(\alpha,\omega_2)\times\{x\}$ 
for all $x\in X$, which then determines the (unique) extension
$g:\omegatwoodelta\times X_0\to\R$ of~$f\circ\vq$.

We show that $g(\omega_2,x)=g(\omega_2,y)$ 
whenever $\vq_{\omega_2}(x)=\vq_{\omega_2}(y)$; for then
$g$~determines a continuous extension of~$f$ to~$X^+$.
We assume $x\neq y$ of course and take disjoint neighbourhoods~$U$ and~$V$
of~$x$ and~$y$ in~$\K_0$.

Using the compactness of~$K_0$ we find two sequences~$\omegaseq\calI$ 
and~$\omegaseq\calJ$ of finite subfamilies of~$\calB_0$
such that the clopen sets 
$I=\bigcup\bigl\{\{n\}\times\bigcup\calI_n:n\in\omega\bigr\}$ and
$J=\bigcup\bigl\{\{n\}\times\bigcup\calJ_n:n\in\omega\bigr\}$
satisfy
\begin{itemize}
\item $I\in x$ and $J\in y$ ($x$ and $y$ are ultrafilters of closed sets), and
\item $I\subseteq U$ and $J\subseteq V$.
\end{itemize}
For each~$n$ let $E_n$ be the set of points in~$K$ that occur as first 
coordinates of endpoints of one of the intervals in~$\calI_n$ and~$\calJ_n$. 
The union, $E$, of these sets is countable.
Therefore there is a~$\beta\ge\alpha$ such that 
$E\cap T_\beta=\emptyset$.
This means that for $\gamma\ge\beta$ the restriction~$q_{\gamma,\omega_2}\restr E$
is injective. 

Because $\vq_{\omega_2}(x)=\vq_{\omega_2}(y)$ the intersection of~$q_{\omega_2}[I]$
and~$q_{\omega_2}[J]$ is not compact.
For every~$n$ and intervals $A\in\calI_n$ and $B\in\calJ_n$ the intersection
of $\vq[A]$ and $\vq[B]$ is contained in~$E_n$.
Therefore $\vq[I]\cap\vq[J]$ is contained in
$F=\bigcup\bigl\{\{n\}\times E_n:n\in\omega\bigr\}$
and hence the common value of $\vq(x)$ and~$\vq(y)$ belongs to $\cl F$.
As the maps $q_{\gamma,\omega_2}$ are injective on~$E$ for $\gamma\ge\beta$, so
are the maps $\vq_{\gamma,\omega_2}$ 
on~$X_\gamma\cap\cl F$ whenever $\gamma\ge\beta$.  

It follows that for all $\gamma\ge\beta$ we have $\vq_\gamma(x)=\vq_\gamma(y)$ 
and therefore 
$$
g(\gamma,x)=f(\gamma,\vq_\gamma(x))=f(\gamma,\vq_\gamma(y))=g(\gamma,y)
$$
and this implies $g(\omega_2,x)=g(\omega_2,y)$, as desired.

\subsection*{$F$-space}
To see that $X^+$~is an $F$-space let $f:X^+\to\R$ be continuous.
We seek a continuous function $k:X^+\to\R$ such that $f=k\cdot|f|$.

As in the proof above we take $\alpha<\omega_2$ such that
$f\circ \vq$~is constant on all horizontal 
lines $(\alpha,\omega_2]\times\{x\}$.

Since $X_{\omega_2}$~is an $F$-space we get a continuous function
$g:X_{\omega_2}\to\R$ such 
that $f(\omega_2,x)=g(x)\cdot\bigl|f(\omega_2,x)\bigr|$.

For all $\beta>\alpha$ we define $k_{\omega_2}$ on $\{\beta\}\times X_\beta$ by
$k_{\omega_2}(\gamma,x)=g(\vq_{\gamma,\omega_2}(x))$,
and $k^+$ on $\{\gamma\}\times X_0$ by $k^+(\gamma,x)=g(\vq_{\omega_2}(x))$.
Then $k^+$~is continuous 
and $k^+=k_{\omega_2}\circ\vq$ on $(\alpha,\omega_2]\times X_0$, 
so that $k_{\omega_2}$~is continuous as well.
Rename $\alpha$ as~$\beta_{\omega_2}$.

Now repeat this argument for every~$\alpha$ 
of cofinality~$\aleph_1$.
First find $\beta_\alpha<\alpha$, as in the proof of Lemma~\ref{lemma.eric},
such that $f\circ\vq$ is constant on $(\beta_\alpha,\alpha]\times\{x\}$
for all~$x$, find a $g$ on~$X_\alpha$ and define~$k_\alpha$ on 
$\{\gamma\}\times X_\gamma$, for $\gamma\in(\beta_\alpha,\alpha]$ as above
by $k_\alpha(\gamma,x)=g(\vq_{\gamma,\alpha}(x))$.

Finally, for every~$\alpha$ of countable cofinality take 
$k_\alpha:\{\alpha\}\times X_\alpha\to\R$ 
such that $f(\alpha,x)=k_\alpha(\alpha,x)\cdot\bigl|f(\alpha,x)\bigr|$
for all~$x$.

Since $\omegatwoodelta$ is Lindel\"of there is a countable subset~$C$ 
of~$\omega_2$ consisting of ordinals of cofinality~$\aleph_1$ such that
the interval $(\beta_{\omega_2},\omega_2]$ together with 
$\bigl\{(\beta_\alpha,\alpha]:\alpha\in C\bigr\}$
covers all but countably many points of~$\omega_2+1$.
From this it is easy to construct a pairwise disjoint clopen cover
of~$\omegatwoodelta$ and combine the various $k_\alpha$ into one continuous
function.

\section{Some variations and questions}
\label{sec.variations}

The construction of our main example admits various variations.

\subsection*{Arbitrarily large covering dimension}

To get a zero-dimensional $F$-space of a prescribed covering dimension~$n$
everywhere in the main construction replace $K_\alpha$ by~$K_\alpha^n$.
Then $\K_{\omega_2}=\beta(\omega\times K^n)$.
By the main result of~\cite{MR0221480} we have $\dim K^n=n$.
The proof of this establishes that the pairs of opposite faces of this
`$n$-cube' form an essential family.
To elaborate: write $\min K=0$ and $\max K=1$ and
for $i\in n$ put $A_i=\{x\in K^n:x_i=0\}$ and $B_i=\{x\in K^n:x_i=1\}$.
Then for every sequence $\langle L_i:i\in n\rangle$ of partitions of~$K^n$, 
with $L_i$ between $A_i$ and $B_i$, the intersection~$\bigcap_{i\in n}L_i$
is nonempty. 
By the Theorem on Partitions, \cite{MR1039321}*{Theorem~7.2.15}, this 
establishes $\dim K^n\ge n$.
In addition \cite{MR0221480}~establishes the 
inequality $\operatorname{Ind}K^n\le n$.
We conclude that $\dim K^n=\operatorname{ind}K^n=\operatorname{Ind}K^n=n$.

To see that $\dim X_{\omega_2}=n$ as well, we consider the projection map
$\pi:\omega\times K^n\to\omega$ and its extension~$\beta\pi$.
In \cite{MR1229130}*{Section~2} it is shown that the components
of~$\K_{\omega_2}$ are exactly the fibers $\beta\pi\preim(u)$ for~$u\in\betaN$.

Next we let $\A_i=\cl(\omega\times A_i)$ and $\B_i=\cl(\omega\times B_i)$
for $i\in n$.
An elementary topological argument will show that for every $u\in\Nstar$
the intersections of the $\A_i$ and $\B_i$ with~$\beta\pi\preim(u)$
form an essential family.
This, together with the equality $\dim\beta Z=\dim Z$ 
(\cite{MR1039321}*{Theorem~7.1.17}), shows that every component 
of~$X_{\omega_2}$ has covering dimension~$n$.

To see that in this case also $\dim X=n$ we first observe
that $\dim X=\dim\beta X\ge\dim X_{\omega_2}=n$.
To get the opposite inequality we let $\calU$ be is a finite open cover 
of~$X^+$.
Its restriction to~$X_{\omega_2}$ has a finite closed refinement of order~$n+1$, 
which can be expanded to a finite family~$\calV$ of open sets that also
has order~$n+1$, covers~$X_{\omega_2}$, and refines~$\calU$.
The argument given below that $\beta X\setminus X_{\omega_2}$ is zero-dimensional
produces an~$\alpha$ such that 
$\bigcup_{\beta>\alpha}\{\beta\}\times X_\beta\subseteq\bigcup\calV$.
The rest of the space, $\bigcup_{\beta\le\alpha}\{\beta\}\times X_\beta$,
is strongly zero-dimension so the restriction of~$\calU$ to this clopen
set has a disjoint open refinement.

This proof will also work if one takes $K_\alpha^\omega$ everywhere, in which
case every component of~$X_{\omega_2}$ will be infinite-dimensional.

\smallskip
Most of the other arguments in Section~\ref{sec.construct} do not rely
on the particular structure of the $K_\alpha$, except the 
proof of $C$-embedding.

One still obtains finite sets $E_m$ of points in~$K$ that occur as end points
of intervals used to create clopen sets in~$\{m\}\times K_0^n$ and hence
in~$\omega\times K_0^n$.
One also still obtains a $\beta\ge\alpha$ with $E\cap T_\beta=\emptyset$.

The set $F$ is replaced by $\bigcup\bigl\{\{m\}\times G_m:m\in\omega\bigr\}$,
where $G_m$~is the grid on $K^n$ defined 
by $\{x:(\exists i\in n)(x_i\in E_m)\}$. 
Then $\vq_{\gamma,\omega_2}$ is injective on $\{\gamma\}\times\cl F$ for all
$\gamma\ge\beta$.

In case $n=\omega$ there is for every $m$ a natural number $k_m$ such that
the supports of the clopen rectangles used in $\{m\}\times K_0^\omega$ are
contained in~$k_m$.
In that case one takes $G_m=\{x:(\exists i\in k_m)(x_i\in E_m)\}$.

\subsection*{Indecomposability}

It is possible to make $X_{\omega_2}$ an indecomposable continuum.

To this end we take a preliminary quotient of every $\omega\times K_\alpha$
by identifying $\orpr n1$ and $\orpr {n+1}0$ for every~$n$
(we still use $0=\min K$ and $1=\max K$).
The result is an infinite string of copies of~$K_\alpha$
and in case $\alpha=\omega_2$ the result is a connected ordered space~$\LL$,
with a minimum, but no maximum.
From a distance it looks like the half~line $\HH=[0,\infty)$ in~$\R$,
with every interval $[n,n+1]$ replaced by~$K$.

The proof that $\HH^*$~is an indecomposable continuum,
see~\cite{MR1229130}*{Section~4}, goes through without any changes
to show that $X_{\omega_2}=\LL^*$ is indecomposable as well.

The proofs that $X$~is zero-dimensional and $C$-embedded in the 
$F$-space~$X^+$ are not affected by these identifications.

We note that if we assume $\neg\CH$ and use $K=[0,1]$ then $X_{\omega_2}$
is actually equal to the remainder~$\HH^*$ of the half line.

\subsection*{Local compactness}

In all variations the space $X_{\omega_2}$ is the sole cause of the failure
of strong zero-dimensionality.

To see this note first that the sets 
$F_\alpha=\bigcup_{\beta>\alpha}\{\beta\}\times X_\beta$ form a neighbourhood
base at~$\{\omega_2\}\times X_{\omega_2}$.
Indeed if $f:X^+\to\R$ is continuous and equal to zero 
on~$\{\omega_2\}\times X_{\omega_2}$ then there is an $\alpha<\omega_2$ such that
$f$~is constant and equal to zero 
on~$F_\alpha$ and the latter is a clopen 
subset of~$X^+$.

Furthermore the complement of such a set, 
$I_\alpha=\bigcup_{\beta\le\alpha}\{\beta\}\times X_\beta$,
is strongly zero-dimensional.
The fastest way to see this is to note that the product
$(\alpha+1)_\delta\times K_0$ is Lindel\"of as a product of a compact
and a Lindel\"of space. 
Therefore $I_\alpha$~is Lindel\"of as well.
By \cite{MR1039321}*{Theorem~6.2.7} the zero-dimensional Lindel\"of 
space~$I_\alpha$ is strongly zero-dimensional.

Therefore the union $Z=\bigcup_{\alpha\in\omega_2}\cl I_\alpha$ is an open
cover of~$\beta X\setminus X_{\omega_2}$ by compact zero-dimensional
sets and hence zero-dimensional.

It follows that $Z$~is a locally compact zero-dimensional $F$-space that is not
strongly zero-dimensional. 

\subsection*{Questions}

Our examples have weight $\aleph_2^{\aleph_0}$, so under $\CH$ the 
$\ZFC$~example cannot be embedded into~$\Nstar$.
We do not know whether it can be embedded if $\CH$ fails.
In fact we do not know the answer to the following question,
which has been asked before but bears repeating often.

\begin{question}
Is there a subspace of $\Nstar$ that is not strongly zero-dimensional?  
\end{question}

\subsubsection*{Weight $\aleph_1$}

It is well-known, and easy to see, that every space of cardinality less 
than~$\cee$ is strongly zero-dimensional.

\smallskip
A similar phenomenon can be observed among $F$-spaces.

If $X$~is an $F$-space and $f:X\to\R$ is continuous then for every $r\in\R$
the closures of $\{x:f(x)<r\}$ and $\{x:f(x)>r\}$ are disjoint and the
complement of the union of these closures is an open set, $O_r$ say.
It follows that if the cellularity of~$X$ is less than~$\cee$ then
there will be many~$r$ such that $O_r=\emptyset$.
For such~$r$ the closures of $\{x:f(x)<r\}$ and $\{x:f(x)>r\}$ would
be complementary clopen sets.
We find that $F$-spaces of cellularity less than~$\cee$ are automatically
strongly zero-dimensional. 
In case its cellularity is countable an $F$-space is even extremally 
disconnected, which means that disjoint open sets have disjoint closures.

\smallskip
What our space leaves unanswered is what happens for $F$-spaces of 
weight~$\aleph_1$.
Of course if $\aleph_1<\cee$ then the comments above show that there
is nothing more to investigate.
Therefore we should assume the Continuum Hypothesis in order to obtain
non-trivial questions and results. 

It has been a rule-of-thumb under the assumption of~$\CH$ that $F$-spaces 
of weight~$\aleph_1$ show many parallels with separable metrizable spaces.
In~\cite{MR2847324} one finds versions for compact $F$-spaces of 
weight~$\aleph_1$ of some well-known theorems for compact metrizable spaces.
In particular that the three main dimension functions coincide on this class.

We ask whether this holds without the compactness condition,
assuming~$\CH$ of course.

\begin{question}
Is every zero-dimensional $F$-space of weight~$\aleph_1$ strongly
zero-di\-men\-sio\-nal?  
\end{question}

And more generally.

\begin{question}
Does the equality $\dim X=\operatorname{ind}X=\operatorname{Ind}X$
hold for every $F$-space of weight~$\aleph_1$?  
\end{question}

\end{document}